\documentclass[12pt]{amsart}
\usepackage{amssymb,amsmath,tabularx,mathrsfs,mathbbol,yfonts,upgreek}
\usepackage{amsthm,verbatim,comment}
\usepackage[bookmarks=true]{hyperref}
\usepackage{pstricks,pst-node,pst-plot}
\usepackage{geometry}
\usepackage{stmaryrd}
\usepackage{paralist}
\usepackage[all]{xy}
\usepackage{array}
\usepackage{url}
\usepackage{longtable}
\numberwithin{equation}{section}

\DeclareSymbolFontAlphabet{\mathbb}{AMSb}
\DeclareSymbolFontAlphabet{\mathbbl}{bbold}

\newtheorem{thm}{Theorem}[section]
\newtheorem{thmx}{Theorem}

\newtheorem{lem}[thm]{Lemma}
\newtheorem{prop}[thm]{Proposition}
\newtheorem{cor}[thm]{Corollary}

\theoremstyle{definition}
\newtheorem{defn}[thm]{Definition}

\newtheorem{eg}[thm]{Example}
\newtheorem{rem}[thm]{Remark}
\theoremstyle{remarks}

\newtheorem*{rem*}{Remarks}

\newtheoremstyle{case}{}{}{}{}{}{:}{ }{}
\theoremstyle{case}

\newcommand{\F}{\mathbb{F}}

\title[On some $p$-transitive association schemes]{On some $p$-transitive association schemes}
\begin{document}
\author{Yu Jiang}
\address[Y. Jiang]{Division of Mathematical Sciences, Nanyang Technological University, SPMS-MAS-05-34, 21 Nanyang Link, Singapore 637371.}
\email[Y. Jiang]{jian0089@e.ntu.edu.sg}


\begin{abstract} In this paper, for any prime $p$, we propose the notion of a $p$-transitive association scheme. This notion aims to generalize the fact that the regular module of a group algebra of a finite group has a unique trivial submodule to the case of the regular modules of modular adjacency algebras. We completely determine the $p$-transitive quasi-thin association schemes and the $p$-transitive association schemes with thin thin residue by their structure theory properties. We also get some results with independent interests.
\vspace{-1em}
\end{abstract}
\maketitle
\noindent{\textbf{Keywords.} Association scheme; Quasi-thin scheme; Scheme with thin thin residue}
\section{Introduction}
Let $X$ be a non-empty finite set and briefly call an association scheme on $X$ a scheme. Let $\F$ be a field of positive characteristic $p$ and denote the $\F$-adjacency algebra of a scheme $S$ by $\F S$. It is well known that $\F S$ controls the structure of $S$. However, compared with the complex adjacency algebra of $S$, very little is known for $\F S$. To understand the structure of $S$ better, it is very necessary to study $\F S$.

For a finite-dimensional $\F$-algebra $A$, the regular $A$-module is very important in studying $A$. For example, $A$ is a semi-simple $\F$-algebra if and only if the regular $A$-module is completely reducible. Therefore we are interested in understanding the structure of the regular $\F S$-module. It is known that the notion of $\F S$ generalizes the notion of the group algebra $\F G$ of a finite group $G$ (see \cite{EI}). We thus want to generalize the results of the regular $\F G$-module to the case of the regular $\F S$-module.

Since the regular $\F G$-module is a transitive permutation module, it always has a unique trivial $\F G$-submodule. The regular $\F S$-module may have more than one trivial $\F S$-submodules (see Example \ref{E;mainexample}). So we propose the notion of a $p$-transitive scheme. Call $S$ a $p$-transitive scheme if the regular $\F S$-module has a unique trivial $\F S$-submodule. This paper focuses on studying the $p$-transitive schemes. Let $O^\vartheta(S)$ (resp. $O_\vartheta(S)$) denote the thin residue (resp. radical) of $S$. For any $T\subseteq S$, let $\langle T\rangle$ denote the closed subset of $S$ generated by $T$. We state our main results as follows.
\begin{thmx}\label{T;A}
Let $S$ be a quasi-thin scheme. Then $S$ is a $p$-transitive scheme if and only if $p>2$ or $p=2$ and $S=O^\vartheta(S)O_\vartheta(S)$ (complex product of $O^\vartheta(S)$ and $O_\vartheta(S)$).
\end{thmx}
\begin{thmx}\label{T;B}
Let $S$ be a scheme with thin thin residue. Then $S$ is a $p$-transitive scheme if and only if $S=\langle S_{p'}\rangle$, where $S_{p'}\subseteq S$ and $S_{p'}$ contains exactly all relations whose valencies are not divisible by $p$.
\end{thmx}
This paper is organized as follows. In Section $2$, we set up the notation and give some required results. We show Theorems \ref{T;A} and \ref{T;B} in Sections $3$ and $4$ respectively.
\section{Notation and preliminaries}
For a general background on association schemes, one may refer to \cite{EI}, \cite{Z2}, or \cite{Z3}. In this section, we set up the notation and present some preliminary results.
\subsection{General conventions}
Throughout this paper, we fix a field $\F$ of positive characteristic $p$ and a non-empty finite set $X$. Let $\mathbb{N}$ denote the set of all natural numbers. For a non-empty set $Y$, let $\langle Y\rangle_\F$ be the $\F$-linear space generated by $Y$. If $Y=\{y\}$, set $\langle y\rangle_\F=\langle \{y\}\rangle_\F$. The addition, multiplication, and scalar multiplication of matrices displayed in this paper are the usual matrix operations. A scheme always means an association scheme on $X$. All modules are finitely generated left modules.
\subsection{Theory of schemes}
Let $S=\{R_0, R_1,\ldots, R_d\}$ be a partition of the cartesian product $X\times X$, where $R_i\neq \varnothing$ for any $0\leq i\leq d$. Then $S$ is called a scheme of class $d$ if the following conditions hold:
\begin{enumerate}[(i)]
\item $R_0=\{(x,x): x\in X\}$;
\item For any $0\leq i\leq d$, we have $0\leq i^*\leq d$, where $\{(x,y): (y,x)\in R_i\}=R_{i^*}\in S$;
\item For any $0\leq i,j,k\leq d$ and $(x,y), (x',y')\in R_k$, the following equality holds: $|\{z\in X: (x,z)\in R_i,\ (z,y)\in R_j\}|=|\{z\in X: (x',z)\in R_i,\ (z,y')\in R_j\}|.$
\end{enumerate}

In this paper, $S=\{R_0, R_1,\ldots, R_d\}$ is a fixed scheme of class $d$. Each member of $S$ is called a relation. By the definition of $S$, for any $0\leq i,j,k\leq d$ and $(x,y)\in R_k$, there is an integer $p_{ij}^k$ such that $|\{z\in X: (x,z)\in R_i,\ (z,y)\in R_j\}|=p_{ij}^k$. For any $0\leq x,y\leq d$, set $k_x=p_{xx^*}^0$ and note that $k_x>0$. It is called the valency of relation $R_x$. Moreover, let $\delta_{ab}$ be the usual Kronecker delta for the integers $a$ and $b$. It is not difficult to get that $p_{0x}^y=p_{x0}^y=\delta_{xy}$.
For any non-empty subsets $U, V, W$ of $S$, set $$UV=\{R_i\in S: \exists\ R_u\in U,\ \exists\ R_v\in V,\ p_{uv}^i>0\}.$$ The operation between $U$ and $V$ is called the complex multiplication of $U$ and $V$. Inductively, for any given $2<m\in \mathbb{N}$ and $m$ non-empty subsets $U_1, \ldots, U_m$ of $S$, the product $U_1\cdots U_m$ with respect to the complex multiplication is defined to be $(U_1\cdots U_{m-1})U_m$. If there exists $0\leq i\leq d$ such that $U_j=\{R_i\}$ for some $1\leq j\leq m$, we always use $R_i$ to replace $U_j$ in the product $U_1\cdots U_m$. For example, if $d\geq 3$, then $\{R_1, R_2\}R_3R_2\{R_2, R_3\}=\{R_1, R_2\}\{R_3\}\{R_2\}\{R_2, R_3\}$. By \cite[Lemma 1.3.1]{Z3}, the complex multiplication is associative. For any $R_k\in S$, $R_0R_k=R_kR_0=\{R_k\}$. If $U\subseteq V$, we also have $UW\subseteq VW$ and $WU\subseteq WV$. We list some needed results.
\begin{lem}\label{L;Intersectionnumber}
For any $0\leq i,j,\ell\leq d$, $S$ has the following properties.
\begin{enumerate}[(i)]
\item [\em (i)] \cite[Lemma 1.1.2 (iii)]{Z3} We have $k_i=k_{i^*}$.
\item [\em (ii)] \cite[Lemma 1.1.3 (iii)]{Z3} We have $\sum_{u=0}^d p_{iu}^j=k_i$.
\item [\em (iii)]\cite[Lemma 1.1.3 (iv)]{Z3} We have $k_ik_j=\sum_{u=0}^dp_{ij}^u k_u$.
\item [\em (iv)]\cite[Proposition 2.2 (vi)]{EI} We have $k_\ell p_{ij}^\ell=k_ip_{\ell j^*}^i=k_jp_{i^*\ell}^j$.
\item [\em (v)]\cite[Lemmas 1.1.2 and 1.5.2]{Z3} We have $p_{ij}^0=\delta_{i^*j}k_i$. Moreover,  $|R_iR_j|$ is no more than the greatest common divisor of $k_i$ and $k_j$.
\item [\em (vi)] \cite[Theorem 1.2.7]{Z2} Assume that $k_i=1$. If $R_iR_j=\{R_\ell\}$, we have $k_j=k_\ell$ and $p_{ij}^\ell=1$. Similarly, if $R_jR_i=\{R_\ell\}$, we have $k_j=k_\ell$ and $p_{ji}^\ell=1$.
\end{enumerate}
\end{lem}
Let $1<m\in \mathbb{N}$ and $\varnothing\neq T\subseteq S$. Let $T^*=\{R_{i^*}: R_i\in T\}$ and $T^m$ be the product $$ \underbrace{T\cdots T}_{\text{m copies}}.$$ By convention, set $T^0=\{R_0\}$ and $T^1=T$. The subset $T$ is called a closed subset of $S$ if $T^*T\subseteq T$. Note that both $\{R_0\}$ and $S$ are closed subsets of $S$. Let $\mathcal{C}$ denote the set of all closed subsets of $S$. By Lemma \ref{L;Intersectionnumber} (iv), $p_{i^*i}^0=k_i>0$ for any $0\leq i\leq d$. So we have $R_0\in T$, $T^*=T$, and $T^2\subseteq T$ if $T\in \mathcal{C}$. Moreover, let $T_1,\ldots, T_m$ be $m$ closed subsets of $S$. Note that $\bigcap_{i=1}^mT_i\in \mathcal{C}$. Let $H\subseteq S$. Use $\langle H\rangle$ to denote $$ \bigcap_{H\subseteq K\in \mathcal{C}}K.$$ Therefore $\langle H\rangle\in \mathcal{C}$. The thin radical of $S$, denoted by $O_\vartheta(S)$, is defined to be the subset $\{R_i\in S: k_i=1\}$ of $S$. Each relation of $O_\vartheta(S)$ is called a thin relation. By Lemma \ref{L;Intersectionnumber} (v) and (vi), notice that $O_\vartheta(S)\in \mathcal{C}$. We need the following result.
\begin{lem}\label{L;closedset}\cite[Lemma 3.1.1 (ii)]{Z3}
Let $\varnothing\neq H\subseteq S$. Then $\langle H\rangle=\bigcup_{i\in \mathbb{N}\cup\{0\}}H^i$.
\end{lem}
Assume further that $T\in \mathcal{C}$. Then $T$ is called a strongly normal closed subset of $S$ if $R_{i^*}TR_i\subseteq T$ for any $0\leq i\leq d$. Note that $S$ itself is a strongly normal closed subset of $S$. Let $\mathcal{D}$ be the set of all strongly normal closed subsets of $S$. Assume further that $T_1,\ldots, T_m\in \mathcal{D}$. Note that $\bigcap_{i=1}^m T_i\in \mathcal{D}$. The thin residue of $S$, denoted by $O^\vartheta(S)$, is defined to be
$$\bigcap_{K\in \mathcal{D}}K.$$
So $O^\vartheta(S)\in \mathcal{D}$. We present a known result of $O^\vartheta(S)$ as follows.
\begin{lem}\label{L;Residue}\cite[Theorem 3.2.1 (ii)]{Z3}
We have $O^\vartheta(S)=\langle \bigcup_{i=0}^d(R_{i^*}R_i)\rangle$.
\end{lem}
We are interested in some special schemes. Let us state their definitions as follows.

The scheme $S$ is called a thin scheme if $S=O_\vartheta(S)$. It is known that every finite group can be viewed as a thin scheme (see \cite[Preface]{Z3}). Therefore the notion of a scheme generalizes the notion of a finite group.

The scheme $S$ is called a quasi-thin scheme if $k_i\leq 2 $ for any $0\leq i\leq d$. As a thin scheme is a quasi-thin scheme, the notion of a quasi-thin scheme generalizes the notion of a thin scheme. The quasi-thin schemes were introduced explicitly by Hirasaka and Muzychuk (see \cite{H2}). They are known to enjoy many good properties. For the details of properties of quasi-thin schemes, one may refer to \cite{H2}, \cite{H3},  and \cite{MP}.

The scheme $S$ is called a scheme with thin thin residue if $O^\vartheta(S)\subseteq O_\vartheta(S)$. As a thin scheme is a scheme with thin thin residue, the notion of a scheme with thin thin residue generalizes the notion of a thin scheme. According to our knowledge, the schemes with thin thin residue were introduced by Zieschang (see \cite{Z1}). They are known to enjoy many good properties. For the details of properties of schemes with thin thin residue, one may refer to \cite{Z1}.

The scheme $S$ is called a $p'$-valenced scheme if $p\nmid k_i$ for any $0\leq i\leq d$. As a thin scheme is a $p'$-valenced scheme, the notion of a $p'$-valenced scheme generalizes the notion of a thin scheme.
\subsection{Modules of schemes}
For any a commutative ring $R$ with identity, let $M_X(R)$ denote the full matrix ring of $R$-matrices whose entries are indexed by the members of $X$. Let $\mathbb{Z}$ be the integer ring. By the definition of $M_X(R)$, the ring $M_X(\mathbb{Z})$ and the $\F$-algebra $M_X(\F)$ are defined. Let $I$ be the identity matrix of $M_X(\mathbb{Z})$. Let $J$ be the all-one matrix of $M_X(\mathbb{Z})$. For any integer $a$, let $\overline{a}\in \F$, where
\[\overline{a}=\begin{cases} \underbrace{1+\cdots+1}_{a\ \text{copies}}, & \text{if}\ a>0,\\
\underbrace{-1+\cdots+(-1)}_{-a\ \text{copies}}, & \text{if}\ a<0,\\
0, & \text{if}\ a=0.
\end{cases}
\]
So $-$ induces a ring homomorphism from $\mathbb{Z}$ to $\F$ by sending every integer $a$ to $\overline{a}$. Moreover, $-$ also induces a ring homomorphism from $M_X(\mathbb{Z})$ to $M_X(\mathbb{\F})$ by sending every matrix $(a_{xy})$ to $(\overline{a_{xy}})$. For any $A=(a_{xy})\in M_X(\mathbb{Z})$, we also write $\overline{A}$ for $(\overline{a_{xy}})$ if there is no confusion.

For any $0\leq i\leq d$, the adjacency matrix with respect to $R_i$, denoted by $A_i$, is a $(0,1)$-matrix $(a_{xy})$ of $M_X(\mathbb{Z})$, where, for any $x, y\in X$, we have $a_{xy}=1$ if and only if $(x,y)\in R_i$.
Let $\F S=\langle \bigcup_{i=0}^d\{\overline{A_i}\}\rangle_\F$. By the definition of $S$, note that $\overline{I}=\overline{A_0}$ and $\overline{J}=\sum_{i=0}^d\overline{A_i}$. So $\bigcup_{i=0}^d\{\overline{A_i}\}$ is an $\F$-basis of $\F S$. Moreover, for any $0\leq i,j\leq d$,
\begin{equation}
A_iA_j=\sum_{k=0}^d p_{ij}^kA_k.
\end{equation}
Therefore $\F S$ is an $\F$-subalgebra of $M_X(\F)$. The $\F$-algebra $\F S$ is called the modular adjacency algebra of $S$. As $\F S$ is an $\F$-algebra, note that $\F S$ itself is an $\F S$-module by left multiplication. This module is called the regular $\F S$-module. For any given one-dimensional $\F$-linear space $\langle v\rangle_\F$, formally define
\begin{equation}
\overline{A_i}v=\overline{k_i}v
\end{equation}
for any $0\leq i\leq d$. By (2.1), (2.2), and Lemma \ref{L;Intersectionnumber} (iii), $\langle v\rangle_\F$ can be regarded as an $\F S$-module. This module is unique up to isomorphism and is called the trivial $\F S$-module. In this paper, we are interested in the following definition.
\begin{defn}\label{D;p-transitive}
The scheme $S$ is said to be a $p$-transitive scheme if the regular $\F S$-module contains a unique trivial $\F S$-submodule.
\end{defn}
Let $G$ be a finite group. Recall that the regular $\F G$-module has a unique trivial $\F G$-submodule. We can generalize this fact for the regular $\F S$-module if $S$ is a $p$-transitive scheme. We collect some basic results of $p$-transitive schemes as follows.
\begin{lem}\label{L;allone}
The scheme $S$ is a $p$-transitive scheme if and only if $\langle \overline{J}\rangle_\F$ is the unique trivial $\F S$-submodule of the regular $\F S$-module. If $p\nmid |X|$, then $S$ is a $p$-transitive scheme. If $\F S$ is a semi-simple $\F$-algebra, then $S$ is a $p$-transitive scheme.
\end{lem}
\begin{proof}
Note that $\langle \overline{J}\rangle_\F$ is a trivial $\F S$-submodule of the regular $\F S$-module. The first assertion thus follows. The second assertion follows by \cite[Proposition 4]{Han1}. The third assertion is from \cite[Proposition 4.1 2]{Han} and the second assertion. We are done.
\end{proof}
For any $v, w\in \F S$, if $v=\sum_{i=0}^dc_i\overline{A_i}$, where $c_i\in \F$ for any $0\leq i\leq d$, we define $\mathrm{Supp}(v)=\{\overline{A_i}\in \F S: c_i\neq 0\}$ and $U(v)=\{R_i\in S: \overline{A_i}\in \mathrm{Supp}(v)\}$. Note that $U(v)=U(w)$ if and only if $\mathrm{Supp}(v)=\mathrm{Supp}(w)$. We need the following lemma.
\begin{lem}\label{L;ResidueRadical}
Let $\langle v\rangle_\F$ be a trivial $\F S$-submodule of the regular $\F S$-module. For any $R_i\in O_\vartheta(S)$, we have $R_iU(v)=U(v)$.
\end{lem}
\begin{proof}
For any $0\leq j\leq d$, since $k_i=1$, by Lemma \ref{L;Intersectionnumber} (v), (vi), and (2.1), we have $R_iR_j=\{R_k\}$ and $\overline{A_iA_j}=\overline{A_k}$. If $j=i^*$, notice that $k=0$ and $\overline{A_iA_{i^*}}=\overline{A_{i^*}A_{i}}=\overline{I}$. Therefore $\overline{A_k}\in \mathrm{Supp}(\overline{A_i}v)$ if $\overline{A_j}\in \mathrm{Supp}(v)$, which implies that $R_iU(v)\subseteq U(\overline{A_i}v)$. Moreover, we also deduce that $R_iR_x\neq R_iR_y$ for any $R_x, R_y\in U(v)$ and $x\neq y$. As $\overline{A_i}v=v$, we have $|R_iU(v)|=|U(v)|=|U(\overline{A_i}v)|$. So $R_iU(v)=U(\overline{A_i}v)=U(v)$.
\end{proof}
We conclude this section by the following lemma.
\begin{lem}\label{L;vectors}
Let $0\neq v=\sum_{i=0}^dc_i\overline{A_i}\in \F S$ and $\overline{A_j}v=\overline{k_j}v$ for any $0\leq j\leq d$.
\begin{enumerate}[(i)]
\item [\em (i)] If $R_0\notin U(v)$, then $R_u\notin U(v)$ for any $0\leq u\leq d$ and $p\nmid k_u$.
\item [\em (ii)] If $R_0\in U(v)$, then $R_u\in U(v)$ and $c_u=c_0$ for any $0\leq u\leq d$ and $p\nmid k_u$.
\item [\em (iii)] If $S$ is a $p'$-valenced scheme, then $S$ is a $p$-transitive scheme.
\end{enumerate}
\end{lem}
\begin{proof}
For (i), if there exists some $0\leq u\leq d$ such that $R_u\in U(v)$ and $p\nmid k_u$, let  $\overline{A_{u^*}}v=\sum_{i=0}^d e_i\overline{A_i}$, where $e_i\in\F$ for any $0\leq i\leq d$. By Lemma \ref{L;Intersectionnumber} (i) and (v), we have $e_0=c_u\overline{k_u}\neq 0$. Moreover, as $p\nmid k_u$, we have $R_0\in U(\overline{A_{u^*}}v)=U(\overline{k_u}v)=U(v)$ by Lemma \ref{L;Intersectionnumber} (i), which contradicts the assumption $R_0\notin U(v)$. (i) thus follows.

For (ii), we may assume that $v\notin\langle \overline{J}\rangle_\F$. Let us put $w=v-c_0\overline{J}$. We have $w\neq 0$. As $R_0\in U(v)$, note that $R_0\notin U(w)$ and $\overline{A_x}w=\overline{k_x}w$ for any $0\leq x\leq d$. By (i), there is no $R_u\in U(w)$ such that $p\nmid k_u$. Therefore we have $R_u\in U(v)$ and $c_u=c_0$ for any $0\leq u\leq d$ and $p\nmid k_u$. (ii) is proved.

For (iii), let $\langle y\rangle_\F$ be a trivial $\F S$-submodule of the regular $\F S$-module. Since $S$ is a $p'$-valenced scheme, by (i) and (ii), note that $R_0\in U(y)$ and $y\in \langle \overline{J}\rangle_\F$. We thus can show (iii) by Lemma \ref{L;allone}. The proof is now complete.
\end{proof}
\section{$p$-Transitive quasi-thin schemes}
In this section, we completely determine all $p$-transitive quasi-thin schemes, which establishes Theorem \ref{T;A}. For our purpose, we first present a required definition.
\begin{defn}\label{D;singularset}
Let $T\subseteq S$. Then $T$ is called a singular subset of $S$ if the following conditions hold:
\begin{enumerate}[(i)]
\item $O_\vartheta(S)\subseteq T$;
\item For any $R_x\in O_\vartheta(S)$, $R_y\in S$, and $R_z\in T$, we have $R_xR_{y^{*}}R_{y}R_z\subseteq T$.
\end{enumerate}
\end{defn}
The following lemma summarizes some properties of the singular subsets of $S$.
\begin{lem}\label{L;singularsubset}
If $T$ is a singular subset of $S$, then the following assertions hold.
\begin{enumerate}[(i)]
\item [\em (i)] We have $O^{\vartheta}(S)\subseteq T$.
\item [\em (ii)] We have $O^\vartheta(S)O_\vartheta(S)\in \mathcal{C}$ and $O^{\vartheta}(S)O_\vartheta(S)$ is a singular subset of $S$.
\item [\em (iii)] We have $O^{\vartheta}(S)O_\vartheta(S)\subseteq T$.
\end{enumerate}
\end{lem}
\begin{proof}
Let $R=\bigcup_{x=0}^d (R_{x^{*}}R_x)$.

For (i), by Lemmas \ref{L;closedset}, \ref{L;Residue}, and Definition \ref{D;singularset} (i), it suffices to check that $R^m\subseteq T$ for any $m\in \mathbb{N}$. We work by induction. For any $0\leq i\leq d$, according to Definition \ref{D;singularset}, notice that $R_{i^*}R_i=R_0R_{i^*}R_iR_0\subseteq T$. So $R^1\subseteq T$. The base case is shown. For any $1<m\in \mathbb{N}$, assume that $R^{m-1}\subseteq T$. For any $R_i\in R^m$, there exist $0\leq j,k\leq d$ such that $R_j\in R$, $R_k\in R^{m-1}$, and $R_i\in R_jR_k$. As $R_j\in R$, there exists $0\leq \ell\leq d$ such that $R_j\in R_{\ell^{*}}R_\ell$. By the inductive hypothesis and Definition \ref{D;singularset} (ii), we have $R_i\in R_jR_k\subseteq R_0R_{\ell^*}R_\ell R_k\subseteq T$, which implies that $R^m\subseteq T$. (i) is shown.

For (ii), let $H\in \mathcal{C}$ and $K\in \mathcal{D}$. We claim that $HK=KH$. For any $R_u\in HK$, as $H^*=H$, there are $0\leq v, w\leq d$ such that $R_{v^*}\in H$, $R_w\in K$, and $R_u\in R_{v^*}R_w$. As $K\in \mathcal{D}$, note that $R_u\in R_{v^*}R_wR_0\subseteq R_{v^*}R_wR_vR_{v^*}\subseteq KR_{v^*}\subseteq KH$. So $HK\subseteq KH$. Similarly, we can check that $KH\subseteq HK$, which implies that $HK=KH$. The claim is shown. According to this claim, we have $KH\in \mathcal{C}$ by \cite[Lemma 2.1.1]{Z3}. As $O^\vartheta(S)\in \mathcal{D}$ and $O_\vartheta(S)\in \mathcal{C}$, we have $O^\vartheta(S)O_\vartheta(S)\in \mathcal{C}$. The first assertion is shown.

For the other assertion, it suffices to check Definition \ref{D;singularset} (ii). For any $R_a\in O_\vartheta(S)$, $R_b\in S$, and $R_c\in R_aR_{b^*}R_b$, note that $R_c\in R_aR_{b^*}R_b\subseteq O_\vartheta(S)O^\vartheta(S)=O^\vartheta(S)O_\vartheta(S)$ by Lemma \ref{L;Residue} and the claim. For any $R_\ell\in O^\vartheta(S)O_\vartheta(S)$, since $O^\vartheta(S)O_\vartheta(S)\in \mathcal{C}$, also note that $R_cR_\ell\subseteq O^\vartheta(S)O_\vartheta(S)$, which implies that $R_aR_{b^*}R_bR_\ell\subseteq O^\vartheta(S)O_\vartheta(S)$. (ii) thus follows.

For (iii), by the claim of (ii), it suffices to show that $O_\vartheta(S)O^\vartheta(S)\subseteq T$. For any $R_r\in O_\vartheta(S)O^\vartheta(S)$, there exist $0\leq s, t\leq d$ such that $R_s\in O_\vartheta(S)$, $R_t\in O^\vartheta(S)$, and $R_r\in R_sR_t$. As $T$ is a singular subset of $S$, we have $R_r\in R_sR_t=R_sR_0R_0R_t\subseteq T$ by (i). So
$O^\vartheta(S)O_\vartheta(S)=O_\vartheta(S)O^\vartheta(S)\subseteq T$ by the claim of (ii). We are done.
\end{proof}
\begin{rem}\label{R;example}
Lemma \ref{L;singularsubset} (ii) and (iii) tell us that $O^\vartheta(S)O_\vartheta(S)$ is the singular subset of $S$ with the smallest cardinality. By Definition \ref{D;singularset}, $S$ itself is the singular subset of $S$ with the largest cardinality.
\end{rem}
\begin{eg}\label{E;singularset}
A singular subset of $S$ may not be a closed subset of $S$. Assume that $S$ is the scheme of order 6, No. 6 in \cite{HM}. Then $S=\{R_0, R_1, R_2, R_3\}$, where $O^\vartheta(S)=O_\vartheta(S)=\{R_0, R_1\}$ and $R_{2^*}=R_3$. In this case, all singular subsets of $S$ are exactly $\{R_0, R_1\}$, $\{R_0, R_1, R_2\}$, $\{R_0, R_1, R_3\}$, and $\{R_0, R_1, R_2, R_3\}$. Notice that $\{R_0, R_1, R_2\}\notin \mathcal{C}$ and $\{R_0, R_1, R_3\}\notin\mathcal{C}$. They are counterexamples.
\end{eg}
The following propositions may have independent interests.
\begin{prop}\label{P;interest}
Assume that $T\in \mathcal{C}$. Then $T$ is a singular subset of $S$ if and only if $O^\vartheta(S)\cup O_\vartheta(S)\subseteq T$.
\end{prop}
\begin{proof}
One direction is clear by Definition \ref{D;singularset} (i) and Lemma \ref{L;singularsubset} (i). For the other direction, as $O^\vartheta(S)\cup O_\vartheta(S)\subseteq T$, it suffices to check that Definition \ref{D;singularset} (ii) holds for $T$. For any $R_x\in O_\vartheta(S)$, $R_y\in S$, and $R_z\in T$, by Lemmas \ref{L;closedset} and \ref{L;Residue}, note that $R_xR_{y^*}R_y\subseteq \langle O^\vartheta(S)\cup O_\vartheta(S)\rangle\subseteq T$ and $R_xR_{y^*}R_yR_z\subseteq T^2\subseteq T$, as desired.
\end{proof}
\begin{prop}\label{P;complete}
Assume that $T$ is a singular subset of $S$. Then $T\in \mathcal{C}$ if and only if $R_{i^*}R_j\subseteq T$ for any $0\leq i\neq j\leq d$ and $R_i, R_j\in T$.
\end{prop}
\begin{proof}
One direction is clear. The other direction is from Lemmas \ref{L;Residue} and \ref{L;singularsubset} (i).
\end{proof}
For any $v\in \F S$, if $v=\sum_{i=0}^d c_i \overline{A_i}$, where $c_i\in \F$ for any $0\leq i\leq d$, recall that $\mathrm{Supp}(v)=\{\overline{A_i}\in\F S: c_i\neq 0\}$ and $U(v)=\{R_i\in S: \overline{A_i}\in \mathrm{Supp}(v)\}$. For our purpose, we need the following lemma.
\begin{lem}\label{L;p-valenced}
Let $T$ be a singular subset of $S$ and $U_T=\{0\leq i\leq d: R_i\in T\}$. Define
$v_T=\sum_{i\in U_T}\overline{A_i}\in \F S$. If $p\mid k_i$ for any $R_i\in S\setminus O_\vartheta(S)$, then $\langle v_T\rangle_\F$ is a trivial $\F S$-submodule of the regular $\F S$-module.
\end{lem}
\begin{proof}
It suffices to check that $\overline{A_a}v_T=\overline{k_a}v_T$ for any given $0\leq a\leq d$. According to the hypotheses, for any $R_i\in S$, we have $\overline{k_i}=1$ or $\overline{k_i}=0$. We thus have two cases.
\begin{enumerate}[\text{Case} 1:]
\item $\overline{k_a}=1$.
\end{enumerate}
In this case, for any $\overline{A_b}\in\mathrm{Supp}(v_T)$, according to Lemma \ref{L;Intersectionnumber} (v), (vi), and $(2.1)$, there exists $0\leq c\leq d$ such that $\overline{A_c}=\overline{A_aA_b}$. Since $k_a=1$, by Lemma \ref{L;Intersectionnumber} (v) and Definition \ref{D;singularset} (ii), $R_c\in R_aR_b=R_aR_0R_b=R_aR_{a^*}R_aR_b\subseteq T,$ which implies that $\overline{A_c}=\overline{A_aA_b}\in \mathrm{Supp}(v_T)$ by the definition of $v_T$. So $\{\overline{A_aA_b}: R_b\in U(v_T)\}\subseteq \mathrm{Supp}(v_T)$. As $\overline{A_a}$ is an invertible matrix, for any $\overline{A_x}, \overline{A_y}\in \mathrm{Supp}(v_T)$, we have $\overline{A_aA_x}=\overline{A_aA_y}$ if and only if $x=y$. Therefore $\mathrm{Supp}(\overline{A_a}v_T)=\{\overline{A_aA_b}: R_b\in U(v_T)\}=\mathrm{Supp}(v_T)$.
By the definition of $v_T$, we have $\overline{A_a}v_T=v_T$ as $\bigcup_{i=0}^d\{\overline{A_i}\}$ is an $\F$-basis of $\F S$.
\begin{enumerate}[\text{Case} 2:]
\item $\overline{k_a}=0$.
\end{enumerate}
In this case, write $\overline{A_a}v_T=\sum_{b=0}^dc_b\overline{A_b}$, where $c_b=\sum_{e\in U_T}\overline{p_{ae}^b}\in\F$ for any $0\leq b\leq d$. For any $0\leq b\leq d$, if there is no $e\in U_T$ such that $p_{ae}^b>0$, then $c_b=0$. If there is some $e\in U_T$ such that $p_{ae}^b>0$, we have $R_{a^*}R_b\subseteq R_{a^*}R_aR_e=R_0R_{a^*}R_aR_e\subseteq T$ as $R_e\in T$ and $T$ is a singular subset of $S$. For any $0\leq u\leq d$, if $p_{au}^b>0$, by Lemma \ref{L;Intersectionnumber} (iv), we have $p_{a^*b}^u>0$, which implies that $R_u\in R_{a^*}R_b\subseteq T$. So we have $u\in U_T$. By Lemma \ref{L;Intersectionnumber} (ii),
$$ c_b=\sum_{e\in U_T}\overline{p_{ae}^b}=\sum_{e=0}^d\overline{p_{ae}^b}=\overline{k_a}=0.$$
Therefore we deduce that $\overline{A_a}v_T=0$. The lemma follows by all listed cases.
\end{proof}
We deduce some corollaries. They may have independent interests.
\begin{cor}\label{C;p-valenced}
If $p\mid k_i$ for any $R_i\in S\setminus O_\vartheta(S)$, then $S$ is a $p$-transitive scheme only if $S=O^\vartheta(S)O_\vartheta(S)$.
\end{cor}
\begin{proof}
Assume that $S$ is a $p$-transitive scheme. By Definition \ref{D;p-transitive} and Lemma \ref{L;p-valenced}, notice that $S$ has precisely one singular subset, which forces that $S=O^\vartheta(S)O_\vartheta(S)$ by Remark \ref{R;example}. The corollary is established.
\end{proof}
\begin{eg}\label{E;Johnson}
Assume that $p\mid k_i$ for any $R_i\in S\setminus O_\vartheta(S)$. The converse of Corollary \ref{C;p-valenced} is not true. For a counterexample, assume that $p=2$ and $S$ is the scheme of order 6, No. 2 in \cite{HM}. We have $S=\{R_0, R_1, R_2\}$. Note that $O_\vartheta(S)=\{R_0, R_1\}$ and $k_2=4$.
Moreover, $O^\vartheta(S)=S$. Note that $\langle \overline{A_0}+\overline{A_1}\rangle_\F$ and $\langle\overline{J}\rangle_\F$ are mutually distinct trivial $\F S$-submodules of the regular $\F S$-module. So $S$ is not a $2$-transitive scheme.
\end{eg}
\begin{cor}\label{C;p-valenced3}
Assume that $p\mid k_i$ for any $R_i\in S\setminus O_\vartheta(S)$. If $S$ is a $p$-transitive scheme, then $S$ is the unique strongly normal closed subset of $S$ that contains $O_\vartheta(S)$.
\end{cor}
\begin{proof}
Let $O_\vartheta(S)\subseteq T\in \mathcal{D}$. By the definition of $O^\vartheta(S)$, note that $O^\vartheta(S)\subseteq T$. By Proposition \ref{P;interest}, observe that $T$ is a singular subset of $S$. According to Lemma \ref{L;singularsubset} (iii), we have $O^\vartheta(S)O_\vartheta(S)\subseteq T\subseteq S$. Since $S$ is a $p$-transitive scheme, the corollary thus follows by Corollary \ref{C;p-valenced}.
\end{proof}
\begin{cor}\label{C;p-valenced2}
Assume that $p\mid k_i$ for any $R_i\in S\setminus O_\vartheta(S)$. If $p\nmid |X|$, then we have $S= O^\vartheta(S)O_\vartheta(S)$. In particular, $S$ is the unique strongly normal closed subset of $S$ that contains $O_\vartheta(S)$.
\end{cor}
\begin{proof}
By Lemma \ref{L;allone}, the assumption $p\nmid |X|$ implies that $S$ is a $p$-transitive scheme. The desired corollary thus follows by Corollaries \ref{C;p-valenced} and \ref{C;p-valenced3}.
\end{proof}
We now work on the quasi-thin schemes. We first offer a known result.
\begin{lem}\label{L;quasi-thincomputation}\cite[Lemma 3.1]{H3}
If $R_a, R_b\in S$ and $k_a=k_b=2$, then precisely one of the following matrix equalities holds.
\begin{enumerate}[(i)]
\item [\em (i)] $A_aA_b=2A_u+2A_v$, where $0\leq u\neq v\leq d$ and $k_u=k_v=1$;
\item [\em (ii)] $A_aA_b=2A_u$, where $0\leq u\leq d$ and $k_u=2$;
\item [\em (iii)] $A_aA_b=2A_u+A_v$, where $0\leq u\neq v\leq d$, $k_u=1$, and $k_v=2$;
\item [\em (iv)] $A_aA_b=A_u+A_v$, where $0\leq u\neq v\leq d$ and $k_u=k_v=2$;
\item [\em (v)] $A_aA_b=A_u$, where $0\leq u\leq d$ and $k_u=4$.
\end{enumerate}
\end{lem}
\begin{lem}\label{L;corelemma}
Assume that $p=2$ and $S$ is a quasi-thin scheme. Let $\langle v\rangle_\F$ be a trivial $\F S$-submodule of the regular $\F S$-module. If we have $R_0\in U(v)$, then $U(v)$ is a singular subset of $S$.
\end{lem}
\begin{proof}
As $R_0\in U(v)$, by Lemma \ref{L;vectors} (ii), notice that $O_\vartheta(S)\subseteq U(v)$. So Definition \ref{D;singularset} (i) is checked. It suffices to check Definition \ref{D;singularset} (ii). As $S$ is a quasi-thin scheme, for any given $0\leq y\leq d$, we have $k_y=1$ or $k_y=2$. For any given $R_x\in O_\vartheta(S)$ and $R_z\in U(v)$, to check the desired containment $$R_xR_{y^*}R_yR_z\subseteq U(v),$$ we distinguish the following cases.
\begin{enumerate}[\text{Case} 1:]
\item $k_y=1$.
\end{enumerate}
In this case, according to Lemma \ref{L;Intersectionnumber} (v), note that $R_xR_z=\{R_f\}$, where $0\leq f\leq d$. By Lemma \ref{L;ResidueRadical}, notice that $R_f\in U(v)$. As $k_y=1$, by Lemma \ref{L;Intersectionnumber} (v) and (vi), $$R_xR_{y^*}R_yR_z=R_xR_0R_z=R_xR_z=\{R_f\}\subseteq U(v).$$
\begin{enumerate}[\text{Case} 2:]
\item $k_z=1$.
\end{enumerate}
In this case, we may assume that $k_y=2$ by Case $1$. Since $k_y=2$, note that $\overline{A_y}v=0.$
By Lemma \ref{L;Intersectionnumber} (v) and (vi), we have $R_yR_z=\{R_w\}$ and $p_{yz}^w=1$, where $0\leq w\leq d$. Since $R_z\in U(v)$, there exists $R_u\in U(v)$ such that $u\neq z$ and $p_{yu}^w>0$. Otherwise, notice that $R_w\in U(\overline{A_y}v)$, which contradicts the fact $\overline{A_y}v=0$. By Lemma \ref{L;Intersectionnumber} (iv), we have $p_{y^*w}^z>0$ and $p_{y^*w}^u>0$, which implies that $R_{y^*}R_w=\{R_z, R_u\}\subseteq U(v)$ by Lemma \ref{L;Intersectionnumber} (i) and (v). So we deduce that
$$ R_xR_{y^*}R_{y}R_z=R_xR_{y^*}R_w=R_x\{R_z, R_u\}\subseteq R_xU(v)=U(v),$$
where the rightmost equality is from Lemma \ref{L;ResidueRadical}.
\begin{enumerate}[\text{Case} 3:]
\item $k_y=k_z=2$ and $A_yA_z=2A_s+2A_t$, where $0\leq s\neq t\leq d$ and $k_s=k_t=1$.
\end{enumerate}
In this case, by $(2.1)$, observe that $p_{yz}^s>0$ and $p_{yz}^t>0$. So $R_yR_z=\{R_s, R_t\}$ by Lemma \ref{L;Intersectionnumber} (v). By Lemma \ref{L;Intersectionnumber} (iv), we have $p_{y^*s}^z>0$ and $p_{y^*t}^z>0$, which implies that $R_{y^*}R_s=\{R_z\}=R_{y^*}R_t$ by Lemma \ref{L;Intersectionnumber} (v). By Lemma \ref{L;ResidueRadical}, we deduce that $$R_xR_{y^*}R_yR_z=R_xR_{y^*}\{R_s, R_t\}=R_xR_z\subseteq R_xU(v)=U(v).$$
\begin{enumerate}[\text{Case} 4:]
\item $k_y=k_z=2$ and $A_yA_z=2A_s$, where $0\leq s\leq d$ and $k_s=2$.
\end{enumerate}
In this case, by $(2.1)$, observe that $R_yR_z=\{R_s\}$ and $p_{yz}^s=2$. As $k_z=k_s=2$, by Lemma \ref{L;Intersectionnumber} (iv), notice that $p_{y^*s}^z=p_{yz}^s=2$, which implies that $R_{y^*}R_s=\{R_z\}$ by Lemma \ref{L;Intersectionnumber} (i) and (iii). By Lemma \ref{L;ResidueRadical}, we have
$$R_xR_{y^*}R_{y}R_z=R_xR_{y^*}R_s=R_xR_z\subseteq R_xU(v)=U(v).$$
\begin{enumerate}[\text{Case} 5:]
\item $k_y=k_z=2$ and $A_yA_z=2A_s+A_t$, where $0\leq s\neq t\leq d$, $k_s=1$, and $k_t=2$.
\end{enumerate}
In this case, by $(2.1)$, notice that $p_{yz}^s=2$ and $p_{yz}^t=1$. So $R_yR_z=\{R_s, R_t\}$ by Lemma \ref{L;Intersectionnumber} (v). Since $p_{yz}^s=2$, by Lemma \ref{L;Intersectionnumber} (iv), we have $p_{y^*s}^z>0$, which implies that $R_{y^*}R_s=\{R_z\}$ by Lemma \ref{L;Intersectionnumber} (v). Since $k_y=2$, observe that $\overline{A_y}v=0$. As we have $p_{yz}^t=1$ and $R_z\in U(v)$, there exists $R_e\in U(v)$ such that $e\neq z$ and $p_{ye}^t>0$. Otherwise, note that $R_t\in U(\overline{A_y}v)$, which contradicts the fact $\overline{A_y}v=0$. By Lemma \ref{L;Intersectionnumber} (iv), we have $p_{y^*t}^z>0$ and $p_{y^*t}^e>0$, which implies that $R_{y^*}R_t=\{R_e, R_z\}$ by Lemma \ref{L;Intersectionnumber} (i) and (v). By Lemma \ref{L;ResidueRadical}, we can deduce that
$$R_xR_{y^*}R_{y}R_z=R_xR_{y^*}\{R_s, R_t\}=R_x\{R_e, R_z\}\subseteq R_xU(v)=U(v).$$
\begin{enumerate}[\text{Case} 6:]
\item $k_y=k_z=2$ and $A_yA_z=A_s+A_t$, where $0\leq s\neq t\leq d$ and $k_s=k_t=2$.
\end{enumerate}
In this case, by $(2.1)$, notice that $p_{yz}^s=p_{yz}^t=1$. So $R_yR_z=\{R_s, R_t\}$ by Lemma \ref{L;Intersectionnumber} (v). Since $k_y=2$, we have $\overline{A_y}v=0$. Since $p_{yz}^s=1$, $s\neq t$, and $R_z\in U(v)$, there exists $R_a\in U(v)$ such that $a\neq z$ and $p_{ya}^s>0$. Otherwise, observe that $R_s\in U(\overline{A_y}v)$, which contradicts the fact $\overline{A_y}v=0$. By Lemma \ref{L;Intersectionnumber} (iv), we have $p_{y^*s}^z>0$ and $p_{y^*s}^a>0$, which implies that $R_{y^*}R_s=\{R_a, R_z\}$ by Lemma \ref{L;Intersectionnumber} (i) and (v). As $p_{yz}^t=1$ and $t\neq s$, let $t$ play the role of $s$ in the proof of above five rows and note that $R_{y^*}R_t=\{R_b, R_z\}$, where $R_b\in U(v)$ and $b\neq z$. By Lemma \ref{L;ResidueRadical}, we can deduce that
$$R_xR_{y^*}R_{y}R_z=R_xR_{y^*}\{R_s, R_t\}=R_x(\{R_a, R_z\}\cup\{R_b, R_z\})\subseteq R_xU(v)=U(v).$$

By Lemma \ref{L;quasi-thincomputation} and all listed cases, Definition \ref{D;singularset} (ii) is checked. We are done.
\end{proof}
The following proposition may have its own interest.
\begin{prop}\label{P;notinterest}
Assume that $p=2$ and $S$ is a quasi-thin scheme. Let $\mathcal{U}$ denote $\{U(v)\subseteq S: R_0\in U(v),\ \overline{A_i}v=\overline{k_i}v\ \forall\ 0\leq i\leq d\}$. Let $\mathcal{S}$ be the set of all singular subsets of $S$. Then $\mathcal{U}=\mathcal{S}$.
\end{prop}
\begin{proof}
By Lemma \ref{L;p-valenced} and Definition \ref{D;singularset} (i), observe that $\mathcal{S}\subseteq\mathcal{U}$. By Lemma \ref{L;corelemma}, we also have $\mathcal{U}\subseteq \mathcal{S}$. The desired proposition follows.
\end{proof}
To state the next lemma, we define $(\F S)^{\F S}=\langle \{v\in \F S: \overline{A_i}v=\overline{k_i}v\ \forall\ 0\leq i\leq d\}\rangle_\F.$ Notice that $\overline{J}\in (\F S)^{\F S}$.
\begin{lem}\label{L;smalllemma}
We have $(\F S)^{\F S}=\langle \{v\in \F S: R_0\in U(v),\  \overline{A_i}v=\overline{k_i}v\ \forall\ 0\leq i\leq d\}\rangle_\F.$
\end{lem}
\begin{proof}
For any $w\in (\F S)^{\F S}$, if $R_0\notin U(w)$, we have $w=x-y$, where $x=\overline{J}$ and $y=\overline{J}-w$. As $R_0\notin U(w)$, note that $R_0\in U(x)\cap U(y)$. The lemma thus follows.
\end{proof}
We use Lemma \ref{L;corelemma} to deduce the following corollary.
\begin{cor}\label{C;corecorollary}
Assume that $p=2$ and $S$ is a quasi-thin scheme. If we have $S=O^\vartheta(S)O_\vartheta(S)$, then $S$ is a $2$-transitive scheme.
\end{cor}
\begin{proof}
By Lemma \ref{L;vectors} (iii), we may assume further that $S$ is not a thin scheme.

For any given trivial $\F S$-submodule $\langle v\rangle_\F$ of the regular $\F S$-module, assume that $v=\sum_{i=0}^dc_i\overline{A_i}$, where $c_i\in \F$. We first show that $v\in \langle \overline{J}\rangle_\F$. We distinguish two cases.
\begin{enumerate}[\text{Case} 1:]
\item $R_0\in U(v)$.
\end{enumerate}
In this case, by Lemma \ref{L;corelemma}, $U(v)$ is a singular subset of $S$. As $S= O^\vartheta(S) O_\vartheta(S)$, by Lemma \ref{L;singularsubset} (iii), we have $S=O^\vartheta(S) O_\vartheta(S)\subseteq U(v)\subseteq S,$ which forces that $U(v)=S$. Therefore $c_i\neq 0$ for any $0\leq i\leq d$.

We claim that $c_i=c_0$ for any $0\leq i\leq d$. By Lemma \ref{L;vectors} (ii), we have $c_j=c_0$ for any $R_j\in O_\vartheta(S)$. Suppose that there is some $0<a\leq d$ such that $R_a\in U(v)$, $k_a=2$, and $c_a\neq c_0$. Notice that $\langle w\rangle_\F$ is also a trivial $\F S$-submodule of the regular $\F S$-module, where $w=v+c_a\overline{J}$. Observe that $R_0\in U(w)$ as $c_a\neq c_0$. By Lemma \ref{L;corelemma}, note that $U(w)$ is a singular subset of $S$. As $S=O^\vartheta(S)O_\vartheta(S)$, by Lemma \ref{L;singularsubset} (iii), we deduce that
$S=O^\vartheta(S)O_\vartheta(S)\subseteq U(w)\subseteq S,$
which forces that $S=U(w)$. However, since $p=2$, we have $R_a\notin U(w)$, which implies that $R_a\notin U(w)=S$. This is an obvious contradiction. The claim is shown. By this claim, we have $c_i=c_0$ for any $0\leq i\leq d$. We thus deduce that $v\in \langle\overline{J}\rangle_\F$.
\begin{enumerate}[\text{Case} 2:]
\item $R_0\notin U(v)$.
\end{enumerate}
In this case, by Lemma \ref{L;smalllemma}, note that $v=x+y$, where we have $x,y\in \F S$ and $R_0\in U(x)\cap U(y)$. Moreover, both $\langle x\rangle_\F$ and $\langle y\rangle_\F$ are trivial $\F S$-submodules of the regular $\F S$-module. We have $v=x+y\in \langle \overline{J}\rangle_\F$ by Case $1$.

The corollary follows by all listed cases and Lemma \ref{L;allone}.
\end{proof}
We are now ready to prove Theorem \ref{T;A}.
\begin{proof}[Proof of Theorem \ref{T;A}]
If $p>2$, $S$ is a $p$-transitive scheme by Lemma \ref{L;vectors} (iii). If $p=2$, by Corollaries \ref{C;p-valenced} and \ref{C;corecorollary}, note that $S$ is a $2$-transitive scheme if and only if $S=O^\vartheta(S)O_\vartheta(S)$. The proof is now complete.
\end{proof}
We close this section by presenting some examples.
\begin{eg}\label{E;mainexample}
Assume that $p=2$. We give some examples of Theorem \ref{T;A}.
\begin{enumerate}[(i)]
\item Assume that $S$ is the scheme in Example \ref{E;singularset}. Observe that $S$ is a quasi-thin scheme with two non-thin relations. We also have $O^\vartheta(S)O_\vartheta(S)=O_\vartheta(S)\neq S$. Therefore $S$ is not a $2$-transitive scheme by Theorem \ref{T;A}.
\item Assume that $S$ is the quasi-thin scheme of order $6$, No. $5$ in \cite{HM}. Observe that $S=\{R_0, R_1, R_2, R_3\}$, where we have $O_\vartheta(S)=\{R_0, R_1\}$, $R_{2^*}R_2=\{R_0,R_2\}$, and $R_1R_2=\{R_3\}$. So $S=O^\vartheta(S)O_\vartheta(S)$. Therefore $S$ is a $2$-transitive scheme by Theorem \ref{T;A}.
\end{enumerate}
\end{eg}
\section{$p$-Transitive schemes with thin thin residue}
In this section, we completely determine all $p$-transitive schemes with thin thin residue. In particular, we finish the proof of Theorem \ref{T;B}. For our purpose, we set $S_{p'}=\{R_i\in S: p\nmid k_i\}$ and recall Definition \ref{D;singularset}. Note that $O_\vartheta(S)\subseteq S_{p'}$.

We first provide some preliminary results.
\begin{lem}\label{L;specialsingular}
If $O^\vartheta(S)\cup S_{p'}\subseteq T\in \mathcal{C}$, then $T$ is a singular subset of $S$.
\begin{proof}
According to the hypotheses, we have $O^\vartheta(S)\cup O_\vartheta(S)\subseteq T$, which implies the desired result by Proposition \ref{P;interest}.
\end{proof}
\end{lem}
\begin{lem}\label{L;necessarypart}
If we have $O^\vartheta(S)\cup S_{p'}\subseteq T\in \mathcal{C}$, $V_T=\{0\leq i\leq d: R_i\in T\}$, and $w_T=\!\sum_{i\in V_T}\overline{A_i}\in \F S$, then $\langle w_T\rangle_\F$ is a trivial $\F S$-submodule of the regular $\F S$-module.
\end{lem}
\begin{proof}
It suffices to check that $\overline{A_a}w_{T}=\overline{k_a}w_{T}$ for any given $0\leq a\leq d$. Let us write $\overline{A_a}w_{T}=\sum_{b=0}^dc_b\overline{A_b},$ where $c_b=\sum_{e\in V_T}\overline{p_{ae}^b}\in\F$ for any $0\leq b\leq d$. We distinguish two cases.
\begin{enumerate}[\text{Case} 1:]
\item $\overline{k_a}=0$.
\end{enumerate}
In this case, for any given $0\leq b\leq d$, if there is no $e\in V_T$ such that $p_{ae}^b>0$,
then $c_b=0$. If there is $e\in V_T$ such that $p_{ae}^b>0$, by Lemma \ref{L;specialsingular} and Definition \ref{D;singularset} (ii), we have $R_{a^*}R_b\subseteq R_0R_{a^*}R_aR_e\subseteq T$. For any $0\leq u\leq d$, if $p_{au}^b>0$, by Lemma \ref{L;Intersectionnumber} (iv), we have $p_{a^*b}^u>0$, which implies that $R_u\in R_{a^*}R_b\subseteq T$. So $u\in V_T$. By Lemma \ref{L;Intersectionnumber} (ii), we thus have
$$ c_b=\sum_{e\in V_T}\overline{p_{ae}^b}=\sum_{e=0}^d\overline{p_{ae}^b}=\overline{k_a}=0.$$
Therefore we deduce that $\overline{A_a}w_{T}=0$.
\begin{enumerate}[\text{Case} 2:]
\item $\overline{k_a}\neq 0$.
\end{enumerate}
In this case, note that $p\nmid k_a$. So $R_a\in S_{p'}\subseteq T$. For any given $0\leq b\leq d$, we claim that $b\in V_T$ if and only if there is some $e\in V_T$ such that $p_{ae}^b>0$. If there is some $e\in V_T$ such that $p_{ae}^b>0$, as $R_a, R_e\in T$ and $T\in \mathcal{C}$, we have $R_b\in R_aR_e\subseteq TT\subseteq T$. So $b\in V_T$. Conversely, if $b\in V_T$, as $R_a\in T$ and $T\in \mathcal{C}$, notice that $R_{a^*}R_b\subseteq T$. So there exists some $e\in V_T$ such that $p_{a^*b}^e>0$. By Lemma \ref{L;Intersectionnumber} (iv), we have $p_{ae}^b>0$. The claim is shown.

By this claim, we have $c_b=0$ if $b\notin V_T$. If $b\in V_T$, by this mentioned claim again, there exists $e\in V_T$ such that $p_{ae}^b>0$. For any $0\leq u\leq d$, if $p_{au}^b>0$, we have $p_{a^*b}^u>0$ by Lemma \ref{L;Intersectionnumber} (iv). We thus have $R_u\in R_{a^*}R_b\subseteq TT\subseteq T$ as $T\in \mathcal{C}$. So we have $u\in V_T$. According to Lemma \ref{L;Intersectionnumber} (ii), we can deduce that
$$c_b=\sum_{e\in V_T}\overline{p_{ae}^b}=\sum_{e=0}^d\overline{p_{ae}^b}=\overline{k_a}.$$
Therefore we have $\overline{A_a}w_{T}=\overline{k_a}w_T$. The lemma follows by all listed cases.
\end{proof}
We deduce the following corollaries. They may have independent interests.
\begin{cor}\label{C;necessarypart}
If $O^\vartheta(S)\subseteq O_\vartheta(S)$, then $S$ is a $p$-transitive scheme only if $S=\langle S_{p'}\rangle$.
\end{cor}
\begin{proof}
Assume that $S$ is a $p$-transitive scheme. Since $O^\vartheta(S)\subseteq O_\vartheta(S)$, by Definition \ref{D;p-transitive} and Lemma \ref{L;necessarypart}, there is a unique closed subset $T$ of $S$ such that $S_{p'}\subseteq T$. Note that $S_{p'}\subseteq\langle S_{p'}\rangle\subseteq S$. We have $S=\langle S_{p'}\rangle$ as $S\in \mathcal{C}$ and $\langle S_{p'}\rangle\in \mathcal{C}$. We are done.
\end{proof}
\begin{cor}\label{C;necessarypart3}
Assume that $O^\vartheta(S)\subseteq O_\vartheta(S)$. If $S$ is a $p$-transitive scheme, then $S$ is the unique closed subset of $S$ that contains $S_{p'}$.
\end{cor}
\begin{proof}
For any $T\in\mathcal{C}$, if $S_{p'}\subseteq T$, notice that $\langle S_{p'}\rangle\subseteq T$. By Corollary \ref{C;necessarypart}, we have $S=\langle S_{p'}\rangle\subseteq T\subseteq S$, which forces that $T=S$. This completes the proof.
\end{proof}
\begin{cor}\label{C;necessarypart2}
Assume that $O^\vartheta(S)\subseteq O_\vartheta(S)$. If we have $p\nmid |X|$, then $S=\langle S_{p'}\rangle$. In particular, $S$ is the unique closed subset of $S$ that contains $S_{p'}$.
\end{cor}
\begin{proof}
According to Lemma \ref{L;allone}, note that the assumption $p\nmid |X|$ implies that $S$ is a $p$-transitive scheme. The corollary thus follows by Corollaries \ref{C;necessarypart} and \ref{C;necessarypart3}.
\end{proof}
\begin{eg}
In Corollary \ref{C;necessarypart2}, the assumption $O^\vartheta(S)\subseteq O_\vartheta(S)$ can not be removed. For a counterexample, assume that $p=2$ and $S$ is the scheme of order 5, No. 2 in \cite{HM}. We have $|X|=5$ and $S=\{R_0, R_1, R_2\}$, where $R_1^2=\{R_0, R_2\}$, $R_2^2=\{R_0, R_1\}$, and $k_1=k_2=2$. According to Lemma \ref{L;Residue}, we have $S=O^\vartheta(S)\nsubseteq O_\vartheta(S)$. Note that $S\neq O_\vartheta(S)=\langle S_{2'}\rangle$. Moreover, both $S$ and $O_\vartheta(S)$ contain $S_{2'}$.
\end{eg}
For any $v\in \F S$, if $v=\sum_{i=0}^d c_i \overline{A_i}$, where $c_i\in \F$ for any $0\leq i\leq d$, recall that $\mathrm{Supp}(v)=\{\overline{A_i}\in\F S: c_i\neq 0\}$ and $U(v)=\{R_i\in S: \overline{A_i}\in \mathrm{Supp}(v)\}$.
\begin{lem}\label{L;corecorelemma}
Assume that $O^\vartheta(S)\subseteq O_\vartheta(S)$. Let $\langle v\rangle_\F$ be a trivial $\F S$-submodule of the regular $\F S$-module. If we have $R_0\in U(v)$, then $\langle S_{p'}\rangle\subseteq U(v)$.
\end{lem}
\begin{proof}
Assume that $v=\sum_{i=0}^dc_i\overline{A_i}$, where $c_i\in \F$ for any $0\leq i\leq d$.

By Lemma \ref{L;closedset}, it suffices to prove that $(S_{p'})^m\subseteq U(v)$ for any $m\in \mathbb{N}$. We work by induction. As $R_0\subseteq U(v)$, by Lemma \ref{L;vectors} (ii), note that $S_{p'}\subseteq U(v)$. The base case is checked. Assume further that $1<m\in \mathbb{N}$ and $(S_{p'})^{m-1}\subseteq U(v)$.

To get a contradiction, suppose that $(S_{p'})^m\nsubseteq U(v)$. Then there exists $0<x\leq d$ such that $R_x\in (S_{p'})^m$ and $R_x\notin U(v)$. So there are $0\leq y,z\leq d$ such that $R_y\in S_{p'}$, $R_z\in (S_{p'})^{m-1}$, and $R_x\in R_yR_z$. In particular, we have $\overline{k_y}\neq 0$ and $p_{yz}^x>0$. As $\langle v\rangle_\F$ is a trivial $\F S$-submodule of the regular $\F S$-module, we have
\begin{equation}
\overline{A_y}v=\overline{k_y}v\neq 0.
\end{equation}
Let us write $\overline{A_y}v=\sum_{a=0}^de_a\overline{A_a}$, where $e_a=\sum_{f=0}^dc_f\overline{p_{yf}^a}\in\F$ for any $0\leq a\leq d$. For any $0\leq u\leq d$, if $p_{yu}^x>0$, by Lemma \ref{L;Intersectionnumber} (iv), we have $p_{y^*x}^u>0$, which implies that $R_u\in R_{y^*}R_x\subseteq R_{y^*}R_yR_z$ as $p_{yz}^x>0$. Since $O^\vartheta(S)\subseteq O_\vartheta(S)$, by Lemma \ref{L;Intersectionnumber} (v), (vi), and $(2.1)$, we have $R_wR_z=\{R_u\}$ and $\overline{A_wA_z}=\overline{A_u}$, where $R_w\in O_\vartheta(S)$. By Lemma \ref{L;Intersectionnumber} (v), (vi), and $(2.1)$ again, for any $0\leq h\leq d$, note that $\overline{A_wA_h}\in(\bigcup_{i=0}^d\{\overline{A_i}\})\setminus\{\overline{A_u}\}$ if $h\neq z$.
As $\langle v\rangle_\F$ is a trivial $\F S$-submodule of the regular $\F S$-module, we thus have
\begin{align}
\sum_{i=0}^d c_i\overline{A_i}=v=\overline{A_w}v=\sum_{i=0}^d c_i\overline{A_w}\overline{A_i}=c_z\overline{A_u}+\mathcal{X},
\end{align}
where $\mathcal{X}\in \F S$ and $R_u\notin U(\mathcal{X})$. Since $\bigcup_{i=0}^d\{\overline{A_i}\}$ is an $\F$-basis of $\F S$, $(4.2)$ tells us that $c_u=c_z$, which
implies that $c_\ell=c_z$ for any $0\leq \ell\leq d$ and $p_{y\ell}^x>0$ by the assumptions of $u$.  Therefore we write $c$ for $c_z$ and deduce that
\begin{equation}
e_x=\sum_{f=0}^dc_f\overline{p_{yf}^x}=\sum_{f=0}^dc\overline{p_{yf}^x}=c(\sum_{f=0}^d\overline{p_{yf}^x})=c\overline{k_y},
\end{equation}
where the rightmost equality is from Lemma \ref{L;Intersectionnumber} (ii). By the inductive hypothesis, we have $R_z\in (S_{p'})^{m-1}\subseteq U(v)$. In particular, we have $c\neq 0$. Recall that we have already known that $\overline{k_y}\neq 0$. So we get $e_x\neq 0$ by $(4.3)$. By the inequality $e_x\neq 0$ and $(4.1)$, we can deduce that
\begin{equation*}
R_x\in U(\overline{A_y}v)=U(\overline{k_y}v)=U(v),
\end{equation*}
which contradicts the assumption $R_x\notin U(v)$. We thus have $(S_{p'})^m\subseteq U(v)$ for any $m\in \mathbb{N}$.
The proof is now complete.
\end{proof}
We use Lemma \ref{L;corecorelemma} to prove the following corollary.
\begin{cor}\label{C;corecorelemma2}
If $O^\vartheta(S)\subseteq O_{\vartheta}(S)$ and $\langle S_{p'}\rangle=S$, then $S$ is a $p$-transitive scheme.
\end{cor}
\begin{proof}
For any given trivial $\F S$-submodule $\langle v\rangle_\F$ of the regular $\F S$-module, we write $v=\sum_{i=0}^d c_i\overline{A_i}$, where $c_i\in\F$ for any $0\leq i\leq d$. By Lemma \ref{L;allone}, it  suffices to show that $v\in \langle \overline{J}\rangle_\F$. We distinguish two cases.
\begin{enumerate}[\text{Case} 1:]
\item $R_0\in U(v)$.
\end{enumerate}
In this case, since $R_0\in U(v)$ and $\langle S_{p'}\rangle=S$, by Lemma \ref{L;corecorelemma}, we can deduce that $S=\langle S_{p'}\rangle\subseteq U(v)\subseteq S$. So $S=U(v)$. We thus have $c_i\neq 0$ for any $0\leq i\leq d$.

It suffices to show that $c_i=c_0$ for any $0\leq i\leq d$. Assume that $c_u\neq c_0$ for some $0\leq u\leq d$. Set $w=v-c_u\overline{J}$ and observe that $\langle w\rangle_\F$ is a trivial $\F S$-submodule of the regular $\F S$-module. Moreover, it is clear that $R_u\notin U(w)$. Since $c_u\neq c_0$, it is also obvious that $R_0\in U(w)$. By Lemma \ref{L;corecorelemma}, we have $R_u\in S=\langle S_{p'}\rangle\subseteq U(w)$, which is a contradiction. So we have $c_i=c_0$ for any $0\leq i\leq d$, which implies that $v\in \langle \overline{J}\rangle_\F$.
\begin{enumerate}[\text{Case} 2:]
\item $R_0\notin U(v)$.
\end{enumerate}
In this case, by Lemma \ref{L;smalllemma}, notice that $v=x+y$, where we have $x,y\in \F S$ and $R_0\in U(x)\cap U(y)$. Moreover, both $\langle x\rangle_\F$ and $\langle y\rangle_\F$ are trivial $\F S$-submodules of the regular $\F S$-module. We thus have $v=x+y\in \langle\overline{J}\rangle_\F$ by Case 1.

The corollary follows by all listed cases.
\end{proof}
Theorem \ref{T;B} is shown by Corollaries \ref{C;necessarypart} and \ref{C;corecorelemma2}. We end this paper by examples.
\begin{eg}\label{E;mainexample22}
Assume that $p=2$. We give some examples of Theorem \ref{T;B}.
\begin{enumerate}[(i)]
\item Assume that $S$ is the scheme in Example \ref{E;singularset}. Notice that $S$ is a scheme with thin thin residue. Moreover, we have $\langle S_{2'}\rangle=O_\vartheta(S)\neq S$. Therefore $S$ is not a $2$-transitive scheme by Theorem \ref{T;B}.
\item Assume that $S$ is the scheme of order 6, No. 4 in \cite{HM}. It is a scheme with thin thin residue. We have $S=\{R_0, R_1, R_2, R_3\}$, where $k_1=k_2=1$ and $k_3=3$. We thus have $S=\langle S_{2'}\rangle$. Therefore $S$ is a $2$-transitive scheme by Theorem \ref{T;B}.
\end{enumerate}
\end{eg}
\subsection*{Acknowledgements}
The author gratefully thanks his Ph.D. supervisor Dr. Kay Jin Lim for organizing the seminar of the theory of association schemes, where all the main results of this paper are motivated and obtained. He also thanks Dr. Kay Jin Lim and Prof. Gang Chen for encouraging him to learn the theory of association schemes. Furthermore, he gratefully thanks Prof. Akihide Hanaki for some helpful comments.

\end{document}